\documentclass[11pt]{amsart}
\usepackage{amssymb}
\usepackage{bm}
\usepackage{graphicx}
\usepackage[centertags]{amsmath}
\usepackage{amsfonts}
\usepackage{amsthm}
\usepackage{a4}
\usepackage{latexsym}
\usepackage{amsmath}
\usepackage{amsfonts}
\usepackage{amssymb}
\usepackage{amscd}
\usepackage{amsthm}
\usepackage{layout}
\usepackage{color}
\newcommand{\longhookrightarrow}{}
\DeclareRobustCommand{\longhookrightarrow}{\lhook\joinrel\relbar\joinrel\rightarrow}

\theoremstyle{plain}% default
\theoremstyle{definition}
\newtheorem{theorem}{Theorem}[section]

\newtheorem{proposition}[theorem]{Proposition}
\newtheorem{corollary}[theorem]{Corollary}
\newtheorem{definition}[theorem]{Definition}
\newtheorem{remark}[theorem]{Remark}

\newtheorem{example}[theorem]{Example}

\newtheorem{problem}{Problem}

\newtheorem*{thm*}{Theorem}

\newtheorem*{dfn*}{Definition}

\newtheorem*{cor*}{Corollary}

\newtheorem*{prp*}{Proposition}

\newtheorem*{rmk*}{Remark}

\newcommand{\R}{\mathbb{R}}

\def\XXint#1#2#3{{\setbox0=\hbox{$#1{#2#3}{\int}$ }
\vcenter{\hbox{$#2#3$ }}\kern-.6\wd0}}

\long\def\symbolfootnote[#1]#2{\begingroup%
\def\thefootnote{\fnsymbol{footnote}}\footnote[#1]{#2}\endgroup}

\begin{document}
\title[Linear extension between Lipschitz maps and Optimal Transport]{Linear extension operators between spaces of\\ Lipschitz maps and Optimal Transport}
\author[L. Ambrosio]{Luigi Ambrosio}
\address{Scuola Normale Superiore, Piazza Cavalieri 7, 56100, Pisa, ITALY}
\email{luigi.ambrosio@sns.it}
\author[D. Puglisi]{Daniele Puglisi}
\address{Dipartimento di Matematca e Informatica, viale A. Doria 6, Catania, ITALY}
\email{dpuglisi@dmi.unict.it}

\date{}

\begin{abstract}
Motivated by the notion of $K$-gentle partition of unity introduced in \cite{LN} and the notion of 
$K$-Lipschitz retract studied in \cite{Ohta}, we study a weaker notion related to
the Kantorovich-Rubinstein transport distance, that we call $K$-random projection.  We
show that $K$-random projections can still be used to provide linear extension operators for Lipschitz maps. We also prove that the existence of these
random projections is necessary and sufficient for the existence of weak$^*$ continuous operators. Finally we use this notion to characterize
the metric spaces $(X,d)$ such that the free space $\mathcal{F}(X)$ has the bounded approximation propriety. 

 %{\color{red}We introduce a new tool in metric space theory, the so called {\em weak $K$-gentle partition of unity}. Throughout this notion, we are able to extend Lipschitz function in a linear and continuous fashion, and to characterize some class of linear Lipschitz extension operator. Moreover, it turns out that this new provides some characterization of metric spaces $(X, d)$ such that the free space $\mathcal{F}(X)$ has the bounded approximation propriety. }
\end{abstract}

\maketitle

\symbolfootnote[0]{\textit{2010 Mathematics Subject
Classification:} Primary 26A16,  47B38}

%\symbolfootnote[0]{\textit{Key words:} }

\symbolfootnote[0]{The second author was supported by ``National Group for Algebraic and Geometric Structures, and their Applications'' (GNSAGA - INDAM).}

\section{Introduction}

Let $(X, d)$ be a metric space, and let $M$ be a non-empty subset of $X$. It is well-known that real-valued Lipschitz functions on $M$ 
can be extended to Lipschitz functions on $X$ with the same Lipschitz constant. Indeed,  in 1934 Mc Shane \cite{McS} observed that 
if $f: M \longrightarrow \mathbb{R}$ is a $L$-Lipschitz function then
\begin{equation}\label{McShane}
E(f)(x) = \inf\{f(m) + L d(x, m):  \ m \in M\}
\end{equation}
defines a $L$-Lipschitz function on $X$ extending $f$ (the largest function with this property). 
This extension formula has some drawbacks. Firstly, the map
$$f \longmapsto E(f)$$
is not linear. Secondly, it relies strongly on the fact that the target space is the real line. Indeed, when the target is infinite-dimensional, it cannot be used to provide Banach-space valued Lipschitz functions.  In \cite{Lind}, Lindenstrauss provided an example in which Banach-space valued Lipschitz functions do not admit extension. Recently, Lee and Naor \cite{LN}, provided a remarkable new method to extend Banach-space valued Lipschitz functions in a bounded linear way, via the so-called $K$-gentle partitions of unity. Among other things, they provided many examples of metric spaces on which there always exist $K$-gentle partitions of unity with respect to any subspace: doubling spaces, negatively curved manifolds, surfaces of bounded genus,... 

In this paper we study a weaker notion, called $K$-random projection on $M$, and prove that this weaker notion is still sufficent to
provide linear extension operators. {\em Strong} $K$-random projections, already introduced in \cite{Ohta} and called $K$-Lipschitz
retracts therein, are families $\{\upsilon_x\}_{x\in X}$ of probability measures in $M$
with finite first moment such that $\upsilon_x=\delta_x$ for
all $x\in M$ and $W_1(\upsilon_x,\upsilon_y)\leq K d(x,y)$ for all $x,\,y\in X$. The quantity $W_1$ used to measure the oscillation of the
$\upsilon_x$ is the well-known {\em Wasserstein distances} or {\em Kantorovich-Rubinstein} duality distance, widely used in Optimal
Transport and in many other fields \cite{Vil}. We prove in Theorem~\ref{thm:main} that every $K$-gentle partition of unity  induces in a natural way a strong $K$-random projection (see also Lemma~4.3 in \cite{Ohta}) and we investigate in Proposition~\ref{prop:cnes} cases when the procedure can be reversed. In this line of thought, it is natural to
define $K$-random projections by requiring only $\upsilon_x$ to be elements of the free space $\mathcal{F}(X)$ (also called
Arens-Eells space), and this weaker notion still
provides linear and weak$^*$ continuous extension operators (Theorem~\ref{extension thm}), when the natural dual topologies on the spaces of Lipschitz functions
are considered. We also prove in Theorem~\ref{characterization} that the existence of weak$^*$-weak$^*$ continuous extension operators is
equivalent to the existence of random projections (see also Proposition~\ref{prop:tuttolip}, in connection with the strong topologies, as well as
the ``finite extension property'' for Lipschitz maps in Corollary~\ref{cor:filter}).

In the final section, see Theorem~\ref{BAP}, we characterize Grothendieck's bounded approximation property of
$\mathcal{F}(X)$ in terms of  the existence of an {\em asymptotic} random projection. We would like to point out that  in the last decade many efforts have been done to establish whether or not $\mathcal{F}(X)$ has  the bounded approximation property. For instance, using the classical Enflo's example of separable Banach space without the approximation property \cite{E}, in \cite{GO} the authors were able to establish the existence of a compact metric space $K$ such that $\mathcal{F}(K)$ fails to have the approximation property.  This method was improved recently in \cite{HLP}. In \cite{HP} it was proved that 
the spaces $\mathcal{F}(\ell_1)$ and $\mathcal{F}(\mathbb{R}^N)$ have even more:  a Schauder basis, improving results previously proved in \cite{LP}. 

\medskip
{\bf Acknowledgement.} The paper was written while the second author was visiting the Scuola Normale Superiore. 
He is grateful to the first author for his kind hospitality. The authors thank A.Naor and D.Zaev for useful bibliographical informations.

\section{Preliminaries}

We shall consider metric spaces $(X,d)$ or even pointed metric spaces $(X,d,\bar x)$, when it will be needed to normalize
the value of Lipschitz functions at a distinguished point $\bar x$.
The {\em doubling constant} of a metric space $(X, d)$, denoted by $\lambda(X)$, 
is the infimum over all natural numbers $\lambda$ such that every ball in $X$ can be covered by $\lambda$ balls of half the radius. 
When $\lambda(X)<\infty$ one says that $(X, d)$ is a {\em doubling metric space}. 

Let $(Z, \|\cdot\|_Z)$ be a Banach space, for a function $f: X \rightarrow Z$ the Lipschitz constant is defined by
\begin{align*}
\|f\|_{Lip} = \sup\{\frac{\|f(x)- f(y)\|_Z}{d(x, y)}: \ x, \,y \in X, \ x \not= y\}.
\end{align*}
$f$ is said to be Lipschitz if $\|f\|_{Lip} < \infty$. For every $M \subseteq X$ closed we denote by $e(M, X, Z)$ the infimum over all constants $K$ such that every Lipschitz function $f: M \rightarrow Z$ can be extended to a Lipschitz function $\widetilde{f}: X \rightarrow Z$ satisfying $\|\widetilde{f}\|_{Lip} \leq K   \|f\|_{Lip}$. We also define 
$$e(X, Z)= \sup\{ e(M, X, Z): \ M \subseteq X\ \hbox{closed}\}$$
and
$$e(X)= \sup\{ e(M, X, Z): \ M\subseteq X\ \hbox{closed}, \ Z \ \hbox{Banach space}\}.$$
It is worth to mention the following remarkable result.

\begin{theorem}[Johnson, Lindenstrauss,  Schechtman \cite{JLS}]
There exists a constant $C>0$ such that for every $n$-dimensional  normed space $X$, $e(X) \leq C n$.
\end{theorem}

In the case when $X$ consists of $n$ points, the $\log n$ upper bound on $e(X)$ from
\cite{JLS} has been improved by Lee and Naor to $\log n/\log(\log n)$, see \cite{Naor} and the references therein for a discussion
about the best lower bounds to date. In this connection, see also \cite{Brudnyi} for the proof of the equivalence of $e(X)$ with the ``linear'' best
extension constant when finite dimensional Banach space targets are considered.

More recently, Lee and Naor have given another surprising result relative to $e(X)$, when doubling metric spaces are considered.

\begin{theorem}[Lee, Naor \cite{LN}]\label{LN1}
There exists a universal constant $C>0$ such that,  for any doubling metric space $(X, d)$,
$$e(X) \leq C \log \lambda(X).$$
\end{theorem}
The previous theorem has been proved using the notion of {\em gentle partition of unity}, whose definition
is recalled below.

\begin{definition}[$K$-gentle partition of unity]
Let $(X, d)$ be a metric space, $M \subseteq X$ a closed subspace and 
$(\Omega, \Sigma, \mathbb{P})$ a measure space. Given $K>0$, we shall say that a function 
$$\varPsi:\Omega \times X \longrightarrow [0, +\infty[$$ 
is a {\em $K$-gentle partition of unity w.r.t. $M$} if the following conditions hold:
\begin{enumerate}
\item[($i$)] for all $x \in M$,  $\varPsi(\cdot, x) \equiv 0$;
\item[($ii$)] for all $x \in X\setminus M$, the $\Sigma$-measurable function $\varPsi(\cdot, x)$ satisfies
$$\int_\Omega\varPsi(\omega,x)\ d\mathbb{P}(\omega) = 1;$$
\item[($iii$)] there exists a $(\Sigma-\mathcal{B}(M))$-measurable function $\gamma: \Omega \longrightarrow M$ such that 
\begin{equation}\label{LN2}
\int_\Omega d(\gamma(\omega), x) \cdot | \varPsi(\omega, x) -  \varPsi(\omega, y)| \ d\mathbb{P}(\omega) \leq K d(x, y)
\qquad\forall x,\,y \in X. 
\end{equation}
\end{enumerate}
\end{definition}

Gentle partitions of unity naturally induce linear and continuous extension operators, which are also monotone in this
sense: if $\langle z^*,f\rangle$ is nonnegative on $M$ for some $z^*\in Z^*$, then the same is true for the extended
map. We reproduce here \cite[Lemma 2.1]{LN}.

\begin{theorem}\label{LN3}
Let $K \geq 1$,  $(X, d)$ be a metric space, $M \subseteq X$ a closed subspace and $Z$ be a Banach space.  
Assume that $\varPsi:\Omega \times X \longrightarrow [0, +\infty[$
is a $K$-gentle partition of unity w.r.t. $M$. Then the extension map (understanding the integral in Pettis' sense)
\begin{align}\label{LN4}
E(f)(x) =
\begin{cases}
f(x) ,\ \hbox{if $x \in M$,}\\
\int_{\Omega} f(\gamma(\omega)) \varPsi(\omega, x) \ d\mathbb{P}(\omega),\ \hbox{if $x \in X \setminus M$}
\end{cases} 
\end{align}
defines a bounded, monotone and linear operator from the space of Lipschitz function from $M$ to $Z$,  to the space of Lipschitz functions  from $X$ to $Z$ with norm less or equal to $K$; namely for every Lipschitz function $f: M \rightarrow Z$, one has $\|E(f)\|_{Lip} \leq K \|f\|_{Lip}$.
\end{theorem}

Let us notice that the proof of the previous theorem relies on the following easy observation: since $\varPsi(\cdot,z)$ is a probability density 
for all $z\in X\setminus M$, for all $x\in X\setminus M$ one has the identity
$$
E(f)(x) - E(f)(y) = \int_\Omega [f(\gamma(\omega)) - f(y)]\cdot [ \varPsi(\omega, x) -  \varPsi(\omega, y)] \ d\mathbb{P}(\omega) 
$$
both when $y\in X\setminus M$ and $y\in M$. Hence, in both cases we can estimate
\begin{equation}\label{LN5}
\begin{split}
\|E(f)(x) - E(f)(y)\| &\leq \|f\|_{Lip}\int_\Omega d(\gamma(\omega),y)|\varPsi(\omega, x) -  \varPsi(\omega, y)| \ d\mathbb{P}(\omega)\\ 
& \leq K \|f\|_{Lip}d(x,y),
\end{split}
\end{equation}
thus getting the desired Lipschitz estimate when at least one of the two points does not belong to $M$
(the case when $x,\,y\in M$ is obvious). On the other hand,
the hard part of the proof of Theorem~\ref{LN1} consists in the proof of the existence, for any $M\subseteq X$ closed, of a
$C \cdot \log(\lambda(X))$-gentle partition of unity w.r.t. to $M$, with $C>0$ universal.

For any pointed metric space $(X,d,\bar x)$ we denote by Lip$_0(X)$ (omitting for notational simplicity the dependence on $d$ and $\bar x$)
the Banach  space of all real-valued Lipschitz functions on $X$ which vanish at $\bar x$, equipped with the natural norm
$$\|f\|_{Lip_0(X)} = \sup \{\frac{|f(x) - f(y)|}{d(x, y)}: \ x \not = y \ \hbox{in} \ X\}.$$ 
An analogous definition can be given for $Z$-valued maps, with $Z$ Banach space.

For all $x \in X$, the Dirac measure $\delta_x$ defines a continuous linear functional on Lip$_0(X)$, defined by $\langle f, \delta_x\rangle = f(x)$,
with $\|\delta_x\|\leq d(x,\bar x)$. Equiboundedness tells us that the closed unit ball of Lip$_0(X)$ is compact for the topology of pointwise convergence on $X$, and therefore the closure of span$\{ \delta_x: \ x \in X\}$ in Lip$_0(X)^*$ is a canonical predual of Lip$_0(X)$, usually denoted by $\mathcal{F}(X)$.  Let us notice that the weak$^*$ topology on $\hbox{Lip}_0(X)$, induced by $\mathcal{F}(X)$ and the topology of pointwise convergence 
induce the same topology on bounded subsets of $\hbox{Lip}_0(X)$.

The  space $\mathcal{F}(X)$ is defined in \cite[Chapter 2]{We} and it is called {\em Arens-Eells space}. However, recently many authors who are studying the geometry of $\mathcal{F}(X)$ call this {\em Free space} associated to $X$. It is sometimes convenient to think of $\mathcal{F}(X)$ as the completion of the set of Borel measures $\mu$ on $X$ with finite support under the norm
$$\|\mu\|_{\mathcal{F}(X)} = \sup_{\|f\|_{Lip_0(X)} \leq 1} \int_X f \ d\mu.$$

We shall use the spaces Lip$_0(X,Z)$ in connection with the existence of extension operators for Lipschitz maps; notice that
if $(X,d,\bar x)$ is a pointed metric space and $M\subseteq X$ is a subspace with $\bar x\in M$, then any extension operator
$E:\hbox{Lip}(M, Z) \rightarrow \hbox{Lip}(X, Z) $ induces by restriction an operator $\tilde{E}:\hbox{Lip}_0(M, Z) \rightarrow \hbox{Lip}_0(X, Z)$;
conversely, any such operator $\tilde E$ can be lifted to an operator $E$ setting
$$
Ef(x):=\tilde{E} (f-f(\bar x))+f(\bar x).
$$
These simple transformations preserve continuity, linearity and monotonicity properties of the operators.

Finally, we stress that the Dirac measure map
$$ \delta: X \longrightarrow \mathcal{F}(X)$$
$$x \longmapsto \delta_x$$
is an isometry.

\subsection{$K$-random projections on closed subsets}

Let $(X, d)$ be a metric space, let $\mathcal{B}(X)$ be its Borel $\sigma$-algebra and let us denote by 
$\mathcal{M}(X)$ the space of all $\sigma$-additive Borel measures with finite total variation, by
$\mathcal{M}_+(X)$ the positive cone. In addition we denote
by $\mathcal{P}_1(X)$ the affine subspace of all Borel probability measures on $X$ with finite
first moment.

For $\mu\in\mathcal{M}(X)$ the total variation measure $|\mu|\in\mathcal{M}_+(X)$ is defined by
\begin{equation}\label{eq:totvar1}
|\mu|(B)=\sup\left\{\sum_i|\mu(B_i)|:\ \text{$B_i$ Borel disjoint partition of $B$}\right\}
\end{equation}
or, equivalently (thanks to the Hahn decomposition) by
\begin{equation}\label{eq:totvar2}
|\mu|(B)=\sup\left\{\mu(C)-\mu(B\setminus C):\ \text{$C\in\mathcal{B}(X)$, $C\subseteq B$}\right\}.
\end{equation}
In the particular case when $\mu(X)=0$ we obtain also
\begin{equation}\label{eq:totvar3}
\frac 12|\mu|(X)=\sup\left\{\mu(C):\ \text{$C\in\mathcal{B}(X)$}\right\}.
\end{equation}

In the following proposition we summarize known relations between $\mathcal{F}(X)$ and $\mathcal{M}(X)$.
We denote by $\mathcal{F}_+(X)$ the class of nondecreasing functionals in $\mathcal{F}(X)$.

\begin{proposition} \label{prop:link} Let $(X,d,\bar x)$ be a pointed metric space.
Any $\mu\in\mathcal{M}(X)$ with $\int_X d(\cdot,\bar x)\ d|\mu|<+\infty$
induces $\upsilon\in\mathcal{F}(X)$, by integration:
\begin{equation}\label{eq:july2}
\upsilon(g):=\int_X g\ d\mu\qquad\forall g\in\hbox{Lip}_0(X)
\end{equation}
and 
\begin{equation}\label{eq:july22}
\|\upsilon\|_{\mathcal{F}(X)}\leq \int_Xd(\cdot,\bar x)\ d|\mu|.
\end{equation}
The converse holds if $(X,d)$ is complete and $\upsilon\in\mathcal{F}_+(X)$, in this case $\upsilon$ is induced by 
$\mu\in\mathcal{P}_1(X)$.
\end{proposition}
\begin{proof}
Clearly, by dominated convergence theorem, the functional in \eqref{eq:july2} is continuous w.r.t. pointwise convergence
on bounded subsets of $\hbox{Lip}_0(X)$. By the Krein-Smulian theorem, it is weak$^*$ continuous, therefore 
$\upsilon\in\mathcal{F}(X)$. Since $|g(x)|\leq d(x,\bar x)$ for all $g\in\hbox{Lip}_0(X)$ with $\|g\|_{Lip}\leq 1$, we
obtain also \eqref{eq:july22}.

The converse statement is proved in \cite{Hil} (see Theorem~3.14 and Theorem~4.3 therein)
for positive functionals in the space $\hbox{Lip}(X)$ of Lipschitz
functions, endowed with the norm $\|f\|_e=\|f\|_{Lip}+|f(\bar x)|$. Let's see how the same property can be achieved 
for nondecreasing functionals in $\mathcal{F}(X)$. Let $\chi(x)=\min\{d(x,\bar x),1\}\in\hbox{Lip}_0(X)$ and let us define
$$
L(f):=\upsilon(\chi f)\qquad f\in\hbox{Lip}(X).
$$
Since $\|f\|_{Lip}\leq\|f\|_{Lip}+\sup_{B_2(\bar x)}|f|\leq 3\|f\|_{Lip}+|f(\bar x)|$,  we can represent
$$
L(f)=\int_X f\ d\mu\qquad \forall f\in\hbox{Lip}(X)
$$
for some $\nu\in\mathcal{M}_+(X)$ with $\mu(X)=\upsilon(\chi)$. Now fix $g\in\hbox{Lip}_0(X)$ bounded nonnegative and,
for $\varepsilon>0$, define
$g_\varepsilon=g\chi/\max\{\chi,\varepsilon\}$, i.e.
$$
g_\varepsilon(x):=
\begin{cases}
g(x) &\text{if $d(x,\bar x)\geq \varepsilon$;}\\ \\
\displaystyle{\frac{g(x)\chi(x)}{\varepsilon}}&\text{if $d(x,\bar x)<\varepsilon$.}
\end{cases}
$$
It is easy to see that $g_\varepsilon\to g$ pointwise as $\varepsilon\to 0^+$ with $\sup_{\varepsilon\in (0,1)}\|g_\varepsilon\|_{Lip}<+\infty$,
hence we can pass to the limit as $\varepsilon\to 0^+$ in
$$
\upsilon(g_\varepsilon)=L(\frac{g}{\chi_\varepsilon})=\int_X \frac{g}{\chi_\varepsilon}\ d\nu
$$
to get \eqref{eq:july2} with $\mu=\chi^{-1}\nu$ in the class of bounded nonnegative $\hbox{Lip}_0(X)$ functions. A simple
approximation then extends the validity of \eqref{eq:july2} to the whole of $\hbox{Lip}_0(X)$.
\end{proof}

Let us recall the {\em Wasserstein distances} (or {\em Kantorovich-Rubinstein distance}) between measures with equal mass. \smallskip

For any measures $\mu,\, \eta \in \mathcal{M}_+(X)$ with $\mu(X)=\eta(X)$,
the Wasserstein distance of order $1$ between $\mu$ and  $\eta$ is defined by 
$$W_1(\mu, \eta) = \inf_{\pi \in \Pi(\mu, \eta)} \int_{X\times X} d(x, y) d\pi(x, y),$$
where $\Pi(\mu, \eta)$ consists of all $\pi\in\mathcal{M}_+(X \times X)$ such that 
$$\pi(A \times X) = \mu(A), \quad \pi(X \times B) = \eta(B), \qquad\ \forall A,\, B\in \mathcal{B}(X).$$
The {\em duality formula}, valid if either $(X,d)$ is Polish or if suitable tightness assumptions are made on $\mu$ and $\eta$, 
is of fundamental importance in many applications:
\begin{equation}\label{eq:defW-1}
W_1(\mu, \eta) = \sup\left\{\int_X g \ d\mu - \int_X g \ d\eta:\ \text{$g\in\hbox{Lip}(X)$, $\|g\|_{Lip}\leq 1$}\right\}.
\end{equation}
Notice that \eqref{eq:defW-1} can also be written for pointed metric spaces $(X,d,\bar x)$ in the form
$$
W_1(\mu, \eta) = \sup\left\{\langle g,\mu-\nu\rangle:\ \text{$g\in\hbox{Lip}_0(X)$, $\|g\|_{Lip_0}\leq 1$}\right\}.
$$
Indeed, in the sequel we shall also consider the case when $\mu$ and $\nu$ belong to the more general
class $\mathcal{F}(X)$, understanding \eqref{eq:defW-1} as a definition, so that $W_1(\mu,\eta)=\|\mu-\eta\|_{\mathcal{F}(X)}$.
In this case it could be that an ``infimum'' representation can still be recovered,
see Remark~\ref{rem:infimum} below  and the partial results in this
direction discussed in \cite{We}. However, this duality will not play a role in our paper.

\begin{remark}\label{rem:infimum}
In a pointed metric space $(X,d,\bar x)$,
it would be interesting to investigate whether the duality formula persists for general $\mu, \,\nu \in \mathcal{F}(X)$, namely
(here we consider $(\bar x,\bar x)$ as the basepoint of $X\times X$)
$$
W_1(\mu,\nu)=\inf\left\{L(d):\ L\in \mathcal{F}(X\times X),\,\,L(g(x)+h(y))=\mu(g)+\nu(h)\right\}.
$$
Notice that, a standard Hahn-Banach procedure seems to require that the marginals are non-negative, a case already
covered by Proposition~\ref{prop:link}.
\end{remark}
 
The concept of strong $K$-random projection introduced below corresponds, to $K$-Lipschitz retracts introduced 
in Definition~3.1 in \cite{Ohta}, when $p=1$.
 
\begin{definition}[$K$-random projections]\label{A1}
Let $(X, d)$ be a metric space, $M \subseteq X$ a closed subspace and $K\geq1$. We shall say that $X$ admits a 
{\em $K$-random projection on $M$} if there a family $\{\upsilon_x:\ x \in X\} \subseteq \mathcal{F}(M)$ such that
\begin{enumerate}
\item[($i$)] For all $\upsilon_x= \delta_x$,  for all $x \in M$;
\item[($ii$)]  for every $x,\, y \in X$, it holds
\begin{equation}\label{A2}
W_1(\upsilon_x, \upsilon_y)\leq K d(x, y). 
\end{equation}
\end{enumerate}
In case $\{\upsilon_x:\ x \in X\} \subseteq \mathcal{P}_1(M)$, we say that $\upsilon_x$ is a strong $K$-random
projection on $M$.
\end{definition}

Of course \eqref{A2} can also be written as $\|\upsilon_x - \upsilon_y\|_{\mathcal{F}(M)}\leq K d(x, y)$. 
Notice also that, thanks to Proposition~\ref{prop:link}, strong $K$-random projections can also be defined by requiring
$\upsilon_x$ to be elements of $\mathcal{F}_+(M)$.

It is obvious that $K$-Lipschitz retraction maps $f:X\longrightarrow M$ correspond to strong $K$-random projections
(the deterministic ones), given by $\upsilon_x=\delta_{f(x)}$.
In the following proposition we provide a basic and simple link between $K$-gentle partitions of unity and strong
$K$-random projections. This was already observed in \cite{Ohta}, see Lemma~4.3 therein, but we provide
a slightly different proof. 
We state the result for Polish spaces, but this assumption is not really restrictive, because
the doubling property is stable under metric completion and complete doubling spaces are proper (i.e. bounded
closed sets are compact) and hence Polish.

\begin{theorem} [$K$-gentle partitions induce strong $K$-random projections]\label{thm:main}
If $(X,d)$ is Polish, any $K$-gentle partition of unity w.r.t. to $M$ induces a strong $K$-random projection on $M$.
In particular, any doubling and Polish metric space $(X, d)$ admits strong
$C\cdot \log(\lambda(X))$-random projections on closed subspaces, for some universal constant $C>0$. 
\end{theorem}	

\begin{proof}
If   $\varPsi:\Omega \times X \longrightarrow [0, +\infty[$ and  $\gamma: \Omega \longrightarrow M$ form a $K$-gentle partition of unity of $X$ 
w.r.t. $M$, then one define the family of probability measures $\{\upsilon_x\}_{x\in X}$ in $M$ as
\begin{align}\label{strong}
\upsilon_x(A):= 
\begin{cases}
\delta_x(A), \ \hbox{if $x \in M$};\\ 
\int_{\gamma^{-1}(A)}  \varPsi(\omega, x)\ d\mathbb{P}(\omega),  \ \hbox{if $x \in X \setminus M$} .
\end{cases}
\end{align}
It is easy to check, using \eqref{LN2} with $y\in M$, that $\upsilon_x$ have finite first moment.
Let us fix $\bar x\in M$, so that $\hbox{Lip}_0(M)$ makes sense, and
let us use the duality formula for the Kantorovich-Rubinstein distance to estimate 
$W_1(\upsilon_x, \upsilon_y)$. Let $g \in \hbox{Lip}_0(M)$ with $\|g\|_{Lip} \leq 1$. If 
$x, \,y \in X \setminus M$ one gets
\begin{align*}
\int_M g \ d\upsilon_x - \int_M g \ d\upsilon_y &= \int_{\Omega} g(\gamma(\omega)) (\varPsi(\omega, x) - \varPsi(\omega, y))\  d\mathbb{P}(\omega)\\
&=  \int_{\Omega} (g(\gamma(\omega))-g(x)) (\varPsi(\omega, x) - \varPsi(\omega, y))\  d\mathbb{P}(\omega)\\
& \leq \int_{\Omega} d(\gamma(\omega), x) |\varPsi(\omega, x) - \varPsi(\omega, y)|\  d\mathbb{P}(\omega)\\
& \leq K d(x, y).
\end{align*}
Therefore, $W_1(\upsilon_x, \upsilon_y) \leq K d(x, y)$. In case $x \in X \setminus M$ and $y \in M$ one has
\begin{align*}
W_1(\upsilon_x, \upsilon_y)&\leq \int_{M \times M} d(z_1, z_2) \ d\upsilon_x \otimes \delta_y(z_1, z_2)\\
&=  \int_{M} d(z_1, y) \ d\upsilon_x(z_1) \\
&=  \int_{\Omega} d(\gamma(\omega), y) \varPsi(\omega, x) \  d\mathbb{P}(\omega)\\
\hbox{(since $ \varPsi(\omega, y) = 0$)}\ &= \int_{\Omega} d(\gamma(\omega), y) (\varPsi(\omega, x) - \varPsi(\omega, y))\  d\mathbb{P}(\omega)\\
&\leq \int_{\Omega} d(\gamma(\omega), y) |\varPsi(\omega, x) - \varPsi(\omega, y)|\  d\mathbb{P}(\omega)\\
& \leq K d(x, y).
\end{align*}
The last case when both $x$ and $y$ belong to $M$ is trivial. 
Therefore, $\{\upsilon_x:\ x \in X\}$ forms a strong $K$-random projection on $M$. 
\end{proof}

In the following proposition we investigate under which conditions the converse holds, namely a strong $K$-random projection
induces a $K$-gentle partition of unity. The proposition shows that basically the difference between the two concepts consists
in the replacement of the distance $W_1$ with the total variation distance, suitably weighted by the distance of the space. 

\begin{proposition}\label{prop:cnes}
Let $\{\upsilon_x:\ x \in X\} \subseteq \mathcal{P}_1(M)$ be a strong $K$-random projection on $M$ and assume that
for some measure $\mu$ on $M$ one has
\begin{equation}\label{eq:a}
\text{$\upsilon_x$ admits density  w.r.t. $\mu$ for each $x \in X\setminus M$.} 
\end{equation}
Then  $\upsilon_x$ induces a $K$-gentle partition of unity w.r.t. $M$ if and only if 
\begin{equation}\label{eq:b} 
\int_Md(z,x)\ d|\upsilon_x-\upsilon_y|(z) \leq K d(x, y)\qquad\forall x,\,y\in X.
\end{equation}
\end{proposition}
\begin{proof}
Assume that the $K$-random projection $\upsilon_x$ is built as in \eqref{strong}, starting from a $K$-gentle partition
$\Psi$, $\gamma$. Fix a Borel set $A\subset M$,
$a>0$, $x,\,y\in X\setminus M$ and notice that \eqref{LN2} gives
\begin{eqnarray*}
\biggl|\int_Ad_a(z,x)\,d(\upsilon_x-\upsilon_y)(z)\biggr|&=&\biggl|
\int_{\gamma^{-1}(A)}d_a(\gamma(\omega),x)(\varPsi(\omega,x)-\varPsi(\omega,y))\ d\mathbb{P}\biggr|
\\&\leq&
\int_{\gamma^{-1}(A)}d(\gamma(\omega),x)|\varPsi(\omega,x)-\varPsi(\omega,y)|\ d\mathbb{P}
\end{eqnarray*}
with $d_a=\min\{d,a\}$.
By taking the supremum w.r.t. $A$ and using \eqref{eq:totvar2} we obtain that the total variation of the measure $d_a(\cdot,x)(\upsilon_x-\upsilon_y)$
can be estimated from above with $Kd(x,y)$, so that
$$
\int_X d_a(z,x)\ d|\upsilon_x-\upsilon_y|(z) \leq K d(x, y).
$$
Eventually by letting $a\uparrow\infty$ we obtain \eqref{eq:b} in this case. If $x\in X\setminus M$ and $y\in M$ the quantity to be estimated
reduces to
$$
\biggl|\int_Ad_a(z,x)\,d\upsilon_x(z)-\chi_A(y)d_a(y,x)\biggr|,
$$
which can be estimated with
$$
\max\biggl\{d_a(x,y), \int_Ad_a(z,x)\,d\upsilon_x(z)\biggr\}.
$$
Then, one can argue similarly as in the previous case, using that $K\geq 1$, \eqref{LN2} and $\varPsi(\omega,y)\equiv 0$. 
In the case $x\in M$ and $y\in X\setminus M$
one needs to estimate
$$
\biggl|\int_Ad_a(z,x)\,d\upsilon_y(z)\biggr|
$$
and this can be done with an analogous argument, still using \eqref{LN2} and $\varPsi(\omega,x)\equiv 0$. Finally, if
both $x$ and $y$ belong to $M$ the inequality is obvious.\smallskip

Conversely,  \eqref{eq:a} tells us that, for every $x \in X \setminus M$, there exists $\varPsi_x \in L_1(M)$ such that
$$\upsilon_x(A) = \int_A \varPsi_x \ d\mu.$$
Thus, one defines 
$$\varPsi: \Omega \times X \longrightarrow [0, +\infty[$$
given by
$$ 
\varPsi(\omega, x)=
\begin{cases}
0, \ \hbox{if $x \in M$}\\
 \varPsi_x(\omega), \ \hbox{if $x \in X \setminus M$}.
\end{cases}
$$
and $\gamma:M \longrightarrow M$ be the identity map. Then, we obviously have: 
\begin{enumerate}
\item[($i$)] for all $x \in M$,  $\varPsi(\cdot, x) = 0$;
\item[($ii$)] $\varPsi(\cdot, x) \in L^1(M,\mu)$ for all $x \in X\setminus M$ and 
$$\|\varPsi(\cdot, x)\|_{L^1} = \int_M \varPsi(\omega, x)\ d\mu(\omega) = \upsilon_x(M) = 1.$$
\end{enumerate}
Since
$$\upsilon_x(A) = \int_{\gamma^{-1}(A)} \varPsi(x, \omega) \ d\mu(\omega),$$
it follows that \eqref{eq:b} corresponds precisely to \eqref{LN2}.
\end{proof}

\begin{remark} Notice that \eqref{eq:a} was used only in the construction of $\Psi$, $\gamma$ from $\upsilon_x$.
It would be interesting to see whether \eqref{eq:a} is really needed, building a ``realization'' of $\upsilon_x$ independently
of this assumption. Of course, if $M$ is countable, then one can 
always choose $\mu = \sum_x \delta_x$ to get a measure satisfying \eqref{eq:a}, and a similar
construction works if $X\setminus M$ is countable.
\end{remark}

Using retraction maps, let us give a simple example showing that the notion of $K$-gentle partition of unity is, 
in general, stronger than the notion of strong $K$-random projection.

\begin{example}\label{main example}
Let $\Gamma$ be an index set with at least countable cardinality and let 
$$\ell_\infty(\Gamma)= \{f:\Gamma \longrightarrow \mathbb{R}: \ f \ \hbox{is bounded}\}$$
which is a Banach space endowed with the norm
$$\|f\|_{\ell_\infty(\Gamma)} = \sup_{\gamma \in \Gamma} |f(\gamma)|.$$
Let us define
$$\ell_1(\Gamma)= \{f \in \ell_\infty(\Gamma): \ \|f\|_{\ell_1(\Gamma)}=\sum_{\gamma \in \Gamma} |f(\gamma)| < \infty\}.$$
%
%that can be viewed as a subspace of
%$$c_0(\Gamma) = \{f \in \ell_\infty(\Gamma): \ \hbox{for every} \ \varepsilon>0 \ \hbox{the set} \ \{\gamma \in \Gamma: \ |f(\gamma)|> \varepsilon\} \ \hbox{is finite}\}.$$
%Let us observe that the support of each element of $c_0(\Gamma)$ is at most countable. 
%
We would like to show the existence of a $1$-Lipschitz retraction $R$ of $X=\ell^+_\infty(\Gamma)$ on
$M=B_{\ell_1^+(\Gamma)}$, where $\ell_\infty^+(\Gamma)$ stands for the positive cone of 
$\ell_\infty(\Gamma)$ and $B_{\ell^+_1(\Gamma)}$ stands for the positive part of the unit ball of $\ell_1(\Gamma)$.
We build $R$ in such a way that $R(X\setminus M)$ contains the uncountable set
$$
M'=\{f\in\ell^+_\infty(\Gamma):\ \sum_{i\in\Gamma}f_i=1\}.
$$
Hence the $1$-projection induced by $R$
does not satisfy the  condition \eqref{eq:b} of Proposition~\ref{prop:cnes}. \smallskip

Indeed, if there were $K \geq 1$ such that  for all $x, \,y \in X$, 
$$\int_Md(z,x)\ d|\upsilon_x-\upsilon_y|(z) \leq K d(x, y)$$
we would have
\begin{align*}
\int_Md(z,x)\ d|\upsilon_x-\upsilon_y|(z) &= \int_M d(z, x) d \delta_{R(x)} +  \int_M d(z, x) d \delta_{R(y)}\\
&= d(x, R(x)) + d(x, R(y)).
\end{align*}
Then $d(x, R(x)) + d(x, R(y)) \leq K d(x, y)$ for any $x,\, y \in X$.
But this occurs only when $X=M$ and $R$ is the identity.

Let  us denote by ${\bf e} \in \ell^+_\infty(\Gamma)$  be the function identically equal to $1$ on $\Gamma$. 
For each $y \in \ell^+_\infty(\Gamma)$ let consider (understanding the positive part componentwise)
$$g(y) = \inf\{t \geq 0: \ \|(y - t {\bf e})^+\|_{\ell_1(\Gamma)} \leq 1 \}.$$

Firstly, we observe that the infimum is attained and that $0 \leq g(y) \leq \|y\|_{\ell_\infty(\Gamma)}$.
Moreover,
\begin{equation}\label{contraction}
|g(y) - g(z)| \leq \|y-z\|_{\ell_\infty(\Gamma)}.
\end{equation}  
Indeed, assume that $g(z) \leq g(y)$, then (understanding the inequalities componentwise)
$$
y - (g(z) +  \|y-z\|_{\ell_\infty(\Gamma)}){\bf e}\leq y - g(z) {\bf e} +z - y 
= z - g(z) {\bf e},
$$
which implies,
$$
 \|[y - (g(z) +  \|y-z\|_{\ell_\infty(\Gamma)}){\bf e}]^+\|_{\ell_1(\Gamma)} 
 \leq \|(z - g(z){\bf e})^+\|_{\ell_1(\Gamma)}
 \leq 1.
$$

By definition, we get
$$g(y) \leq g(z) +  \|y-z\|_{\ell_\infty(\Gamma)}.$$
Using the map $g$, and taking \eqref{contraction} into account, we are able to define a $1$-Lipschitz retraction of 
 $\ell_\infty^+(\Gamma)$ on $B_{\ell_1^+(\Gamma)}$ given by $y\mapsto (y - g(y){\bf e})^+$.
\end{example}

Let us recall that a metric space $(X, d)$ is said to be uniformly discrete if there exists $\varepsilon>0$ such that
$$d(x, y) \geq \varepsilon, \quad \forall x,\, y \in X, \ x \not = y.$$
Following  \cite[p. 138]{JLS} we have the following
\begin{proposition}
Assume that $M$ is a uniformly discrete subset of $(X, d)$, and that $M$ has finite diameter $D$. 
Then $X$ admits a $\frac{2D}{\varepsilon}$-random projection on $M$.
\end{proposition}
\begin{proof}
Let us suppose for simplicity that $D=1$.
Fix $t_0 \in M$ and define, for every $x \in X$,
$$
\upsilon_x=
\begin{cases}
\frac{2}{\varepsilon} \left[d(x, t) \delta_{t_0} +(\frac{\varepsilon}{2} - d(x, t) )\delta_t \right], \ \hbox{if $x \in B(t, \frac{\varepsilon}{2})$
 for some $t\in M$,}\\
{}\\
\delta_{t_0},  \ \hbox{if $x \not \in \bigcup_{t \in M} B(t, \frac{\varepsilon}{2})$}.
\end{cases}
$$
It is simple to verify that $\{\upsilon_x: \ x \in X\} \subseteq \mathcal{P}_1(M)$ satisfies
\begin{enumerate}
\item[] $\upsilon_x = \delta_x$, whenever $x \in M$;
\item[] $W_1(\upsilon_x, \upsilon_y) \leq \frac{2}{\varepsilon} d(x, y)$, for all $x, y \in X$.
\end{enumerate}
\end{proof}

\begin{remark}
After reading a preliminary version of this paper, A. Naor kindly pointed out to us that if $M$ is uniformly discrete subset 
of bounded diameter of a metric space $X$, then $X$ admits even a $K$-gentle partition of unity w.r.t. $M$ (with $K=O(D/ \epsilon)$). 
\end{remark}

Now, we are ready to state the Lipschitz extension theorem relative to $K$-random projections, which follows
closely Theorem~\ref{LN3}. Fixing the value of the function at a common basepoint $\bar x$ for $M$ and $X$,
we shall work with the $\hbox{Lip}_0$ spaces, which came in natural duality with the AE spaces.

\begin{theorem}\label{extension thm}
Let $(X, d,\bar x)$ be a pointed metric space, $M \subseteq X$ a closed subspace with $\bar x\in M$ and $Z$ be a Banach space. 
Assume there exists $\{\upsilon_x: x \in X\}$ a $K$-random projection on $M$. Then there exists a 
bounded, and linear extension operator
$$E: \hbox{Lip}_0(M, Z) \longrightarrow  \hbox{Lip}_0(X, Z)$$
with $\|E\| \leq K$. In addition $E$ is monotone if $\upsilon_x\in\mathcal{F}_+(M)$ for all $x\in X$.
\end{theorem}

\begin{proof}
Let us define $E: \hbox{Lip}_0(M, Z) \longrightarrow  \hbox{Lip}_0(X, Z)$ by 
\begin{align}\label{A3}
\langle z^*,E(f)(x)\rangle = \upsilon_x(\langle z^*,f\rangle)
\qquad\forall z^*\in Z^*,\,\,\forall x\in X
\end{align}
which corresponds, when $\upsilon_x$ are measures, to the definition of $E(f)$ as a Pettis integral.
%Firstly, observe that $E(f)(x)$ is well defined as Pettis integral. Indeed, for every $y \in M$ and $a>0$, the Lipschitz property
%of the bounded function $s\mapsto\min\{\|f(s)\|,a\}$ give
%\begin{align*}
%&\int_M \min\{\|f(s)\|,a\} d\upsilon_x(s) \\
%&\leq \left\{\int_M \min\{\{\|f(s)\|,a\} d\upsilon_x(s)
%\int_M \min\{\{\|f(s)\|,a\} d\upsilon_y(s) \right\} + \|f(y)\| \\
%& \leq \|f\|_{Lip} W_1(\upsilon_x, \upsilon_y) + \|f(y)\| \\
%& \leq \|f\|_{Lip}K d(x, y)+ \|f(y)\| < \infty
%\end{align*}
%and we can let $a\uparrow\infty$ to obtain that $\int_M\|f\|d\upsilon_x<+\infty$. 
Of course $E(f)$ extends $f$, 
since $\upsilon_x = \delta_x$ for all points $x \in M$. It remains to show that $E(f)$ is also Lipschitz. Let us fix $z^*$ 
in the unit ball of $Z^*$ and observe that
\begin{align*}
z^*(E(f)(x)-E(f)(y)) &= z^*(E(f)(x)) - z^*(E(f)(y)) \\
& = \upsilon_x( \langle z^*,f\rangle)  -  \upsilon_y(\langle z^*,f\rangle)\\
& \leq \|z^*(f)\|_{Lip} W_1(\upsilon_x, \upsilon_y)  \\
& \leq \|f\|_{Lip}K d(x, y), 
\end{align*}
which implies that $\|E(f)(x)-E(f)(y)\| \leq  \|f\|_{Lip}K d(x, y)$. 
\end{proof}

Let us come back for a moment to the Lee-Naor extension Theorem \ref{LN3}. In case $X$ admits a $K$-random projection on $M$, 
then the operator
$$E: \hbox{Lip}_0(M) \longrightarrow  \hbox{Lip}_0(X)$$
given in (\ref{A3}) when $Z=\R$ defines a bounded linear operator extension with $\|E\| \leq K$. A closer look at $E$ shows, 
that $E$ is continuous from $\hbox{Lip}_0(M)$ endowed with the weak$^*$ topology to $\hbox{Lip}_0(X)$
endowed with the topology of pointwise convergence.  Therefore, recalling that the weak$^*$ topology coincides
on bounded sets with the topology of pointwise convergence, the restriction of $E$ 
to bounded sets is weak$^*$-weak$^*$ continuous. By the Krein-Smulian theorem it follows that $E$ is weak$^*$-weak$^*$ continuous.

Now, by a duality argument, we characterize when this kind of extension operators may exist.

\begin{theorem}\label{characterization}
Let $(X, d,\bar x)$ be a pointed metric space, $M \subseteq X$ a closed subspace with $\bar x\in M$. 
Then the following properties are equivalent:
\begin{enumerate}
\item[($a$)]  there exists  $\{\upsilon_x: x \in X\}$ a $K$-random projection on $M$, for some $K \geq 1$;
\item[($b$)] there exists a bounded linear extension operator
$$E: \hbox{Lip}_0(M) \longrightarrow  \hbox{Lip}_0(X)$$
such that
\begin{enumerate}
\item[($i$)] $\|E\| \leq K$; 
\item[($ii$)] $E$ is weak$^*$-weak$^*$ continuous.
\end{enumerate}
\end{enumerate}
\end{theorem}

\begin{proof}
$(a) \Rightarrow (b)$ This part is a particular case of Theorem~\ref{extension thm}.  Formula (\ref{A3}) defines a bounded linear extension 
operator on  $\hbox{Lip}_0(M)$ which satisfies (i) and (ii), thanks to the comments made after Theorem~\ref{extension thm}.

$(b) \Rightarrow (a)$ Let $E: \hbox{Lip}_0(M) \longrightarrow  \hbox{Lip}_0(X)$
be a bounded linear extension operator as in ($b$). Thanks to property (ii), we can find a bounded linear operator 
$$S:\mathcal{F}(X) \longrightarrow \mathcal{F}(M)$$
with $\|S\| \leq K$ such that $S^* = E$  (see for instance Theorem~3.1.11 in \cite{Me}).

Then,  one defines
$$\upsilon_x = S(\delta_x)  \quad \hbox{for each $x \in X$}.$$
%Firstly, let us observe that $\upsilon_x$ is a probability measure on $M$. Indeed,
%$$
%\upsilon_x(M)  = \langle c_M^1, \upsilon_x \rangle 
%=  \langle E(c_M^1), \delta_x \rangle 
% =  \langle c_X^1, \delta_x \rangle =1.
%$$
For every $x \in M$ and $g \in \hbox{Lip}_0(M)$ we have
$\langle g, \upsilon_x \rangle = \langle E(g), \delta_x \rangle= g(x)$,
therefore $\upsilon_x = \delta_x$. Finally,
\begin{align*}
W_1(\upsilon_x , \upsilon_y) &= \sup_{\|g\|_{Lip_0}\leq 1} \left\{S\circ\delta_x(g) - S\circ\delta_y(g)\right\}\\
&=  \sup_{\|g\|_{Lip_0}\leq 1}\left\{ E(g)(x)-E(g)(y)\right\}\\
\hbox{(since $\|E(g)\|_{Lip} \leq K$)} \ & \leq K W_1(\delta_x, \delta_y)\\
&= K d(x, y).
\end{align*}
 \end{proof}

Before closing this section, we would like to add some word regarding the existence of bounded linear extension operators among 
spaces of Lipschitz maps. Let $(X,d,\bar x)$ be any pointed metric space, let $M$ be a closed subset with $\bar x\in M$ and 
let us denote by $\hbox{Lip}_0(M)^*$ the dual of $\hbox{Lip}_0(M)$. Therefore we can consider the Dirac measure map
$$\delta : M \longhookrightarrow  \hbox{Lip}_0(M)^*.$$
Thus, one natural question could also be if $\delta$ can be extended as Lipschitz function on the whole space $X$. 
This somehow corresponds to a ``very weak'' random projection, where $\upsilon_x$ are in the dual of $\hbox{Lip}_0(M)$.

\begin{proposition}\label{prop:tuttolip}
Let $(X,d,\bar x)$ be any pointed metric space and let $M$ be a closed subset with $\bar x\in M$. Then the following properties are equivalent:
\begin{enumerate}
\item[($a$)]  there exists a bounded linear extension operator
$$E: \hbox{Lip}_0(M) \longrightarrow  \hbox{Lip}_0(X);$$
\item[($b$)] there exists a Lipschitz map $\hat\delta: X \longrightarrow \hbox{Lip}_0(M)^*$ such that the following diagram commutes
$$
\begin{array}{ccc}
M& \overset{\delta}{\longhookrightarrow} &  \hbox{Lip}_0(M)^*\\
\Big\downarrow&\overset{\hat\delta}{\nearrow}&\\
X&&
\end{array}
$$
\end{enumerate}
In addition, for any $\hat\delta$ as in (b) one has $\|E\| = \|\hat\delta\|$.
\end{proposition}

\begin{proof}
$(a) \Rightarrow (b)$ It is enough to define $\hat\delta: X \longrightarrow \hbox{Lip}_0(M)^*$ by
$$\hat\delta(x) = E^*(\delta_x) \quad \forall x \in X.$$
Let us first observe that $\hat\delta(x) \in \hbox{Lip}_0(M)^*$ for every $x \in X$. Linearity follows directly by the definition of $\hat\delta$. Moreover, 
\begin{align*}
\|\hat\delta(x)\|_{Lip_0(M)^*} &= \sup_{\|g\|_{Lip_0(M)} \leq 1} \langle g, \hat\delta(x) \rangle\\
&=  \sup_{\|g\|_{Lip_0(M)} \leq 1} \langle E(g), \delta_x \rangle \\
& \leq   \|E\|d(x,\bar x).
\end{align*}
Of course, $\hat\delta(x) = \delta_x$ for every $x \in M$. Finally, since $\|Eg\|_{Lip_0(X)}\leq\|E\|\|g\|_{Lip_0(M)}$
for all $g\in\hbox{Lip}_0(M)$, we get
\begin{align*}
\|\hat\delta(x) - \hat\delta(y) \|_{Lip_0(M)^*}&=  \sup_{\|g\|_{Lip_0(M)} \leq 1} |\langle g, \hat\delta(x) -  \hat\delta(y)\rangle|\\
&=  \sup_{\|g\|_{Lip_0(M)} \leq 1}  |E(g)(x) -  E(g)(y)|\\
&\leq \|E\|  d(x, y).
\end{align*}
It follows that
\begin{equation}\label{eq:stima1}
\|\hat\delta\|\leq\|E\|.
\end{equation}
$(b) \Rightarrow (a)$  Let us define $E: \hbox{Lip}_0(M) \longrightarrow  \hbox{Lip}_0(X)$ by
$$E(g)(x) = \langle\hat\delta(x) , g\rangle \quad \forall x \in X,\, \forall g \in \hbox{Lip}_0(M).$$
Of course $E$ is a bounded linear operator. Let us estimate its norm: for 
$g \in  \hbox{Lip}_0(M)$ with $\|g\|_{Lip_0} \leq 1$ we have
\begin{align*}
|E(g)(x) - E(g)(y)| &= | \langle \hat\delta(x) , g\rangle - \langle \hat\delta(y) , g\rangle |\\
&=  | \langle \hat\delta(x) -  \hat\delta(y) , g\rangle |\\
& \leq \| \hat\delta(x) - \hat\delta(y) \|_{Lip_0(M)^*}\\
& \leq \|\hat\delta\|  d(x, y).
\end{align*} 
Thus 
\begin{equation}\label{eq:stima2}
\|E\| \leq \|\hat\delta\|.
\end{equation} 
Finally,  $E(g)(x) = \langle\hat\delta(x) , g\rangle
=\langle \delta(x) , g\rangle
=g(x)$ for all $x\in M$.
Therefore, $E$ is an extension operator and \eqref{eq:stima1} and \eqref{eq:stima2} give $\|E\|=\|\hat\delta\|$.
\end{proof} 

An immediate consequence of the previous Proposition is the finite extension property for Lipschitz maps. 
In the proof we use a classical ``finite dimension'' extension argument which resembles
the proof of Theorem~A in \cite{Brudnyi}.

\begin{corollary}\label{cor:filter}
Let $(X,d,\bar x)$ be a pointed metric space and let $M \subseteq X$ be a closed subspace with $\bar x\in M$
and with the following {\em finite extension Lipschitz property}:
\begin{enumerate}
\item[(F)] For every $F \subseteq M$ finite with $\bar x\in F$ there exists a linear extension operator 
$$E_F: \hbox{Lip}_0(F) \longrightarrow  \hbox{Lip}_0(X)$$
with $\|E_F\| \leq C$, for some constant independent of $F$.
\end{enumerate}
Then, there exists a linear extension operator 
$$E: \hbox{Lip}_0(M) \longrightarrow  \hbox{Lip}_0(X)$$
with $\|E\| \leq C$.
\end{corollary}

\begin{proof}
Firstly, let us notice that if $R_F: \hbox{Lip}_0(M) \longrightarrow \hbox{Lip}_0(F)$ denotes the restriction operator, 
since $R_F$ is continuous and surjective the dual operator $$R_F^* : \hbox{Lip}_0(F)^* \longhookrightarrow \hbox{Lip}_0(M)^*$$ 
is continuous and injective, hence its range is a closed subspace of $\hbox{Lip}_0(M)^*$.

By the previous proposition, we can translate the hypothesis by the following: 
for every $F \subseteq M$ finite with $\bar x\in F$ there exists a Lipschitz map 
$f_{F}: X \longrightarrow \hbox{Lip}_0(F)^*$ such that the following diagram commutes
$$
\begin{array}{ccc}
F& \stackrel{\delta}{\longhookrightarrow} &  \hbox{Lip}_0(F)^*\\
\Big\downarrow&\stackrel{f_F}{\nearrow}&\\
X&&
\end{array}
$$

This tells us that 
\begin{enumerate}
\item[($i$)] $f_F(x) = \delta_x$ for every $x \in F$;
\item[($ii$)] $\|f_F(x) - f_F(y)\|_{Lip_0(F)^*} \leq C d(x, y)$ for every $x,\,y\in F$.
\end{enumerate}
Still by the previous proposition, we need to build a Lipschitz map $f:X \longrightarrow \hbox{Lip}_0(M)^*$ such that the diagram
$$
\begin{array}{ccc}
M& \stackrel{\delta}{\longhookrightarrow} &  \hbox{Lip}_0(M)^*\\
\Big\downarrow&\stackrel{f}{\nearrow}&\\
X&&
\end{array}
$$
commutes.

Let us denote by $B_F(r)=\{x^* \in \hbox{Lip}_0(F)^*: \ \|x^*\|_{Lip_0(F)^*}  \leq r\}$ be the closed ball in $\hbox{Lip}_0(F)^*$ centered at $0$ with radius $r>0$. Since $\hbox{Lip}_0(F)^*$  is finite dimensional, each ball $B_F(r)$ is a compact set, and then, by the natural embeddings $R_F^*$,  $B_F(r)$ can be seen as a compact subset of $\hbox{Lip}_0(M)^*$.

In particular, ($ii$) implies that 
$$f_F \in \prod_{x \in X} B_F(C d(x,\bar x)) =B \subseteq ( \hbox{Lip}_0(M)^*, \hbox{weak}^*)^X.$$
When we partially order the collection $\mathcal{F}$ of finite subsets of $M$ by inclusion we have a net; hence, by the compactness of
$B$ in $(\hbox{Lip}_0(M)^*, \hbox{weak}^*)^X$, there exist a cofinal subnet $\mathcal{G}$ and $f:X \longrightarrow \hbox{Lip}_0(M)^*$
such that
$$\lim_{F\in \mathcal{G}} f_F = f\qquad\text{in $(\hbox{Lip}_0(M)^*, \hbox{weak}^*)^X$}.$$
Now, by cofinality, for each $x \in M$ there exists $F \in \mathcal{G}$ such that $x \in F$. Since the convergence is in weak$^*$ topology, 
in particular we have pointwise convergence. Thus ($i$) implies that 
$$\delta_x=\lim_{F\in\mathcal{G}} f_F(x) = f(x), \quad \forall x \in M.$$

Similarly, for every $x,\, y \in X$, by ($ii$), 
\begin{align*}
\|f(x) - f(y)\|_{Lip_0(M)^*} &= \lim_F \|f_F(x) - f_F(y)\|_{Lip_0(M)^*} \\
& \leq C d(x, y).
\end{align*}
Thus, $f$ is Lipschitz and $\|f\|_{Lip} \leq C$.
 \end{proof}

\subsection{Bounded Approximation Property}

It is time to  recall Grothendieck's bounded approximation property. Let $1 \leq \lambda < \infty$.  A Banach space $X$ has the $\lambda$-bounded approximation property ($\lambda$-BAP) if, for every $\varepsilon > 0$ and every compact set $K \subseteq X$, there is a bounded finite-rank linear operator $T :X\longrightarrow X$ with
\begin{enumerate}
\item[($i$)] $\|T\| \leq \lambda$;
\item[($ii$)] $\|T(x) - x\| \leq \varepsilon$ whenever $x\in K$.
\end{enumerate}
We say that $X$ has the BAP if it has the $\lambda$-BAP, for some $1 \leq \lambda < \infty$.

Given a Banach space $E$ and $F$ a closed subspace, we recall that the annihilator $F^\perp$ of $F$ is defined as
$$F^\perp = \{ x^* \in E^*: \ x^*(x) = 0 \ \forall x \in F\}.$$ 
We say that $F$ is an $M$-ideal in $E$ if there is a linear projection $P:E^* \longrightarrow E^*$ (i.e. $P^2=P$) such that $Range P= F^\perp$ and $\|x^*\| = \|P(x^*)\| + \|x^* - P(x^*)\|$ for all $x^* \in E^*$. Thus in this case $E^* = Ker P \oplus F^\perp$. Let us recall the following useful theorem (see 
for instance \cite[p. 59]{HWW}).

\begin{theorem}[Ando-Choi-Effros]
Suppose $F$ is a $M$-ideal of a Banach space $E$, $Z$ is a separable Banach space with $\lambda$-BAP  and $T: Z \longrightarrow E/F$ be a bounded linear operator with $\|T\|=1$. Then there exists a continuous linear lifting of $T$; i.e. there is $L:Z \longrightarrow E$ with $\|L\| \leq \lambda$ such that
$$Q \circ L = T,$$
where $Q: E \longrightarrow E/F$ denotes the quotient map.
\end{theorem}

Since $\mathcal{F}(X)$ is a Banach space it is natural to ask whether this space has the BAP. 
We provide  a  characterization of separable metric spaces such that this happens. In particular, it follows from Theorem \ref{thm:main} that for any separable doubling metric space $(X, d, \bar x)$, $\mathcal{F}(X)$ has BAP. This was already observed in  \cite{LP} using  the notion of  gentle partition of unity. Our proof relies on an  asymptotic  argument.

For the next theorem, let $(X,d,\bar x)$ be a  pointed  separable metric space, let $\{x_n: n \in \mathbb{N}\}$ be a countable  set dense in $X$ 
with $x_0=\bar x$ and 
let  $M_n = \{x_0, x_1, \dots, x_n\}$, for all $n \in \mathbb{N}$.

\begin{theorem}\label{BAP}
Let $(X,d,\bar x)$ be  a pointed separable metric space and let $x_n,\, M_n$ as above. Then the following properties are equivalent:
\begin{enumerate}
\item[($a$)] $\mathcal{F}(X)$ has $K$-BAP;
\item[($b$)] there exists $K \geq 1$ such that $X$ admits a {\em asymptotic} $K$-random projection on $(M_n)_n$; 
i.e., for every $n \in \mathbb{N}$ there exist $\{\upsilon_x^n: x \in X\} \subseteq \mathcal{F}(M_n)$ such that:
\begin{enumerate}
\item[($i$)] $\lim_n W_1(\upsilon_x^n, \delta_x) = 0$ for every $x \in \bigcup_n M_n$;
\item[($ii$)] $W_1(\upsilon_x^n, \upsilon_y^n) \leq K d(x, y)$, for every $x,\, y \in X$ and $n \in \mathbb{N}$.
\end{enumerate}
\end{enumerate}
\end{theorem}

\begin{proof} Let us denote by
$$R_n: \hbox{Lip}_0(X) \longrightarrow   \hbox{Lip}_0(M_n)$$
the restriction map. Then $R_n$ is a bounded linear operator with $\|R_n\| \leq 1$.  
Moreover, the operator $\hat{R}_n: \mathcal{F}(M_n)	\hookrightarrow \mathcal{F}(X)$ defined by
$$
\hat{R}_n\mu(f)=\mu(R_n f)
$$
satisfies $\hat{R}_n^* = R_n$ and is an isometry, since any function $f\in \hbox{Lip}_0(M_n)$ is the restriction of a function
in $\hbox{Lip}_0(X)$ with the same Lipschitz constant. Hence, in the sequel, we naturally consider $\mathcal{F}(M_n)$ as a closed subspace of $\mathcal{F}(X)$.

$(a) \Rightarrow (b)$ We start with a standard argument, whose idea relies on the proof of Ando-Choi-Effros' theorem itself. 

Let us consider the space
$$\mathcal{C}= \{(\mu_n)_n: \ \mu_n \in \mathcal{F}(M_n), \  (\mu_n)_n \ \hbox{is norm convergent in} \ \mathcal{F}(X) \}$$
endowed with the supremum norm; here we use the canonical embeddings of $ \mathcal{F}(M_n)$ into $\mathcal{F}(X)$
which are predual to the restriction operator $R_n$.
Let us denote by $\mathcal{C}_0$ the subspace of all sequences $(\mu_n)_n$ converging to zero in $\mathcal{F}(X)$. Then, it is standard to see that  $\mathcal{C}_0$ is an $M$-ideal of $\mathcal{C}$.  According to Theorem~2.2(iv) of \cite{HWW}, this can be proved if
we check that $\mathcal{C}_0$  have the so-called 3-ball property in $\mathcal{C}$, namely for any $(\mu_n)_n$ in the unit ball of
$\mathcal{C}$, any $\varepsilon>0$ and any $(\nu^i_n)_n$ in the unit ball of $\mathcal{C}_0$, $i=1,2,3$, there is $(\tilde\mu_n)_n\in\mathcal{C}_0$ with
$$
\|(\mu_n)_n+(\nu^i_n)_n-(\tilde\mu_n)_n\|\leq 1+\varepsilon.
$$
It is easy to check that if we define $\tilde\mu_n=\mu_n$ for $1\leq n\leq N$ and $\tilde\mu_n=0$ for $n>N$ this property
holds, if $N$ is large enough.

%Moreover, since  $\bigcup_n M_n$ is dense in $X$, it follows that the canonical operator from $\mathcal{C}$ to $\mathcal{F}(X)$, which maps each sequence $(\mu_n)_n$ to its norm-limit, is also surjective. Since the kernel of this operator coincides with $\mathcal{C}_0$, the open mapping theorem implies that the quotient space $\mathcal{C}/\mathcal{C}_0$ is isomorphic 
%\footnote{{\color{blue} Non mi e' chiaro perche' open mapping implica che il quoziente e' isometrico; d'altro canto mi sembra che giochi un ruolo nel seguito, quindi forse nonandrebbe messo tra parentesi}} 
%to $\mathcal{F}(X)$.

Let us observe that the limit operator $L:\mathcal{C}/\mathcal{C}_0\hookrightarrow\mathcal{F}(X)$ induces
an isomorphism between $\mathcal{C}/\mathcal{C}_0$ and $\mathcal{F}(X)$. Indeed, since we are considering the quotient
with respect to $\mathcal{C}_0$, it is clear that $L$ is injective, and since $\mathcal{C}$ is endowed with the sup norm
one has $\|L\|\leq 1$.
%
%Indeed, let $(\mu_n)_n \in \mathcal{C}$ with limit $\mu \in \mathcal{F}(X)$. Since $\|\mu\|_{\mathcal{F}(X)}>0$ we can assume that $\|\mu_n\|_{\mathcal{F}(M_n)} >0$ for each $n \in \mathbb{N}$. 
%Let us define  $\widetilde{\mu}_n \in \mathcal{F}(M_n)$ by 
%$$\widetilde{\mu}_n = \frac{\|\mu\|}{\|\mu_n\|} \mu_n, \quad \forall n \in \mathbb{N}.$$
%Since 
%\begin{align*}
%\|\widetilde{\mu}_n - \mu_n\|_{\mathcal{F}(X)} = \left|\frac{\|\mu\|}{\|\mu_n\|} - 1\right| \|\mu_n\|_{\mathcal{F}(M_n)} \stackrel{n \rightarrow \infty}{\longrightarrow} 0,
%\end{align*}
%
%we get that  $(\widetilde{\mu}_n)_n$ is in the same equivalence class of $(\mu_n)_n$ with 
%$$\|\mu\|_{\mathcal{F}(X)} = \|(\widetilde{\mu}_n)_n\|_{\mathcal{C}} = \sup_n \|\widetilde{\mu}_n\|_{\mathcal{F}(M_n)}.$$
%Therefore, $\|(\mu_n)_n\|_{\mathcal{C}/\mathcal{C}_0} \leq \|\mu\|_{\mathcal{F}(X)}$. On the other hand, for the other inequality, it is enough to observe that, for any $(\mu_n)_n \in \mathcal{C}$ with limit $\mu \in \mathcal{F}(X)$, and $(\nu_n)_n \in \mathcal{C}_0$, since 
%$$\|\mu\| = \lim_{n \rightarrow \infty} \|\mu_n + \nu_n\|  \leq \sup_n \|\mu_n + \nu_n\|.$$
%It follows that $\|\mu\|_{\mathcal{F}(X)} \leq \|(\mu_n)_n\|_{\mathcal{C}/\mathcal{C}_0}$.\smallskip
%
It remains to show that any element in $\mu \in \mathcal{F}(X)\setminus\{0\}$ is the limit of $ (\mu_n)_n \in \mathcal{C}$ with
$\sup_n\|\mu_n\|=\|\mu\|$.
If $\mu \in \hbox{span}\{ \delta_x: \ x \in X\}$, with $\mu= \sum_{i=1}^k a_i \delta_{x_i}$, then by density of $\cup_n M_n$ in $X$, for each $i=1, \ldots, k$ 
we can find $x^n_i\in M_n$ convergent as $n\to\infty$ to $x_i$. Then, using the definition of the norm in $\mathcal F(X)$, it is easily seen that
$$
\tilde\mu_n:= \sum_{i=1}^k a_i \delta_{x^n_i}\in\mathcal F(M_n)
$$
converge to $\mu$, so that if we define $\mu_n=\|\mu\|\tilde\mu_n/\|\tilde\mu_n\|$ for $n$ large enough and $\mu_n=\delta_{x_0}$
for finitely many $n$, we are done.
Since $\mathcal{F}(X)$ is the completion of the set of Borel measures on X with finite support, if $\mu \in \mathcal{F}(X)$, there exists $(\nu_k)_k \subseteq \hbox{span}\{ \delta_x: \ x \in X\}$ such that $\nu_k \longrightarrow \mu$ in $\mathcal{F}(X)$. By a diagonal argument, since all $\nu_k$ can
be approximated, we can find a strictly increasing family of indeces $n_k\to\infty$ and $\sigma_k\in\mathcal F(M_{n_k})$ such that 
$\sigma_k\to\mu$ in $\mathcal{F}(X)$. Since the spaces $\mathcal F(M_n)$ are nested, if we define
$$
\mu_n:=\sigma_k\quad\text{for $n_k\leq n<n_{k+1}$,}
$$
we obtain $\mu_n\in\mathcal F(M_n)$ and $\mu_n\to\mu$ in $\mathcal{F}(X)$.

Thus, identifying $\mathcal{C}/\mathcal{C}_0$ with $\mathcal F(X)$,
we are in position to apply  Ando-Choi-Effros' theorem for the identity operator $Id_{\mathcal{F}(X)}$ on $\mathcal{F}(X)$. 
There exists a bounded linear operator
$$L:\mathcal{F}(X) \longrightarrow\mathcal{C}$$ 
with $\|L\| \leq K$ (since $\mathcal{F}(X)$ has $K$-BAP) such that
\begin{equation}\label{right inverse}
Q \circ L = Id_{\mathcal{F}(X)},
\end{equation}
where $Q: \mathcal{C}\longrightarrow  \mathcal{F}(X)$ denotes the quotient map.\smallskip

For each $n \in \mathbb{N}$, let
$$\pi_n: \mathcal{C} \longrightarrow \mathcal{F}(M_n)$$
be the canonical $n$-th coordinate projection and consider the composition
$$\pi_n \circ L : \mathcal{F}(X) \longrightarrow \mathcal{F}(M_n).$$
Then $\pi_n \circ L$ is a bounded linear operator with $\|\pi_n \circ L\| \leq \lambda$.

Now let us define  an asymptotic $K$-random projection on $(M_n)_n$ in the following way: 
$$\upsilon_x^n = (\pi_n \circ L) (\delta_x), \quad \forall x \in X, \ n \in \mathbb{N}.$$
Then $\{\upsilon_x^n: x \in X\} \subseteq \mathcal{F}(M_n)$  and satisfy
\begin{enumerate}
\item[($i$)] $\lim_n W_1(\upsilon_x^n, \delta_x) = 0$, for every $x \in \bigcup_n M_n$;
\item[($ii$)] $W_1(\upsilon_x^n , \upsilon_y^n)  \leq K d(x, y)$ for all $x, y \in X$. 
\end{enumerate}
For ($i$), let $n_0 \in \mathbb{N}$ such that $x \in M_{n_0}$, thus $\delta_x \in \mathcal{F}(M_n)$ for all $n \geq n_0$. Then, by the construction above, we can identify $\delta_x \in \mathcal{F}(X)$ in the quotient  space with the equivalence class  generated by
$$ (0, 0, \dots, \underbrace{\delta_x}_{n_0-th}, \delta_x, \delta_x , \dots) \in  \mathcal{C}/ \mathcal{C}_0.$$
Thanks to (\ref{right inverse}), the sequences 
$$ 
(0, 0, \dots, \underbrace{\delta_x}_{n_0-th}, \delta_x, \delta_x , \dots) \ \hbox{and} \ (\pi_1(L(\delta_x)), \dots, \pi_n(L(\delta_x)), \dots)
$$
have to be in the same equivalence class; i.e. 
$$\lim_n \pi_n( L(\delta_x)) = \delta_x \ \hbox{in} \ \mathcal{F}(X).$$
Therefore
\begin{align*}
\lim_n W_1(\upsilon_x^n, \delta_x) &=  \lim_n \sup_{\|g\|_{Lip_0(X)} \leq 1} \left\{\upsilon_x^n(g) - g(x)\right\} \\
&= \lim_n \sup_{\|g\|_{Lip_0(X)} \leq 1}  \langle g, \pi_n\circ L(\delta_x)  - \delta_x\rangle\\
&= \lim_n \|\pi_n( L(\delta_x)) - \delta_x \|_{\mathcal{F}(X)}=0.
\end{align*}

For ($ii$), let $x, \,y \in X$. Then,
\begin{align*}
W_1(\upsilon_x^n , \upsilon_y^n) &= \sup_{\|g\|_{Lip_0(M_n)}\leq 1} \left\{\upsilon_x^n(g) - \upsilon_y^n(g)\right\}\\
&=  \sup_{\|g\|_{Lip_0(M_n)}\leq 1} \left\{ \pi_n\circ L (\delta_x)(g)  - \pi_n\circ L (\delta_y)(g)\right\}\\
&=  \|\pi_n\circ L(\delta_x)-\pi_n\circ L(\delta_y)\|_{\mathcal{F}(M_n)}\\
& \leq \lambda \|\delta_x-\delta_y\|_{\mathcal{F}(X)}= \lambda d(x, y).
\end{align*}
$(b) \Rightarrow (a)$ Let us denote by  $\{\upsilon_x^n:\ x \in X\} \subseteq \mathcal{F}(M_n)$ be an asymptotic $K$-random
projection on $(M_n)_n$. As in Theorem~\ref{extension thm}, one can define 
$E_n: \hbox{Lip}_0(M_n) \longrightarrow   \hbox{Lip}_0(X)$ by
$$E_n(f)(x) = \langle\upsilon_x^n,f\rangle.$$
Then $E_n$ is a bounded linear  operator with $\|E_n\| \leq K$. Finally, we define
$$S_n: \hbox{Lip}_0(X) \longrightarrow \hbox{Lip}_0(X)$$
simply by composing
$$S_n = E_n \circ R_n.$$
Since $\hbox{Lip}_0(M_n)$ is a finite dimensional space, $(S_n)_n$ is a sequence of finite-rank linear operators with $\|S_n\| \leq K$,  which are weak$^*$-weak$^*$ continuous on bounded subsets of $\hbox{Lip}_0(X)$.

Let us fix $x \in X$, $f \in  \hbox{Lip}_0(X)$ and $\varepsilon>0$. Since $\{x_n: n \in \mathbb{N}\}$ is dense, there exists $n_0 \in \mathbb{N}$ such that $d(x, x_{n_0}) \leq \varepsilon$. Thanks to the hypothesis ($i$) on $\upsilon_x^n$'s, we can find $n_1 \geq n_0$ such that for all $n \geq n_1$, $W_1(\upsilon_{x_{n_0}}^n, \delta_{x_{n_0}}) \leq \varepsilon$.

Then, if $n \geq n_1$ one has
\begin{align*}
&|f(x) - f(x_{n_0})| \leq \varepsilon \|f\|_{Lip},\\
& |S_n(f)(x_{n_0}) - f(x_{n_0})| \leq \|f\|_{Lip} W_1(\upsilon_{x_{n_0}}^n, \delta_{x_{n_0}}) \leq  \varepsilon  \|f\|_{Lip};\\
& |S_n(f)(x) - S_n(f)(x_{n_0})| \leq \|f\|_{Lip} W_1(\upsilon_x^n, \upsilon_{x_{n_0}}^n) \leq \varepsilon  K\|f\|_{Lip} ,
\end{align*}
which imply that
$$|S_n(f)(x) - f(x)| \leq (2 + K) \varepsilon \|f\|_{Lip}.$$
This argument shows that $(S_n(f))_n$ pointwise converges  to $f$.

Let 
$$T_n: \mathcal{F}(X)  \longrightarrow \mathcal{F}(X)$$
 bounded linear operators such that  $T_n^* = S_n$. It turns out that $T_n$ is a finite rank operator with $\|T_n\| \leq K$, for every $n\in \mathbb{N}$.  Moreover, $(T_n)_n$ converges to the identity w.r.t. the weak operator topology on  $\mathcal{F}(X) $. 
 
Since $(X, d)$ is separable and the Dirac measure map $\delta: X \longrightarrow \mathcal{F}(X)$ is an isometry with the span of the range dense, we also have that $\mathcal{F}(X)$ is separable. As in \cite[Proposition 2.1]{K}, by using the classical Mazur's theorem, convex combinations of the operators $T_n$ and a standard diagonal argument yield 
the existence of a sequence of finite-rank operators converging to the identity for the strong operator topology on $\mathcal{F}(X)$. 
This sequence of finite rank operators will be also uniformly bounded in norm by $K$, since it arises from convex combinations of uniformly bounded operators. It is easy to check that the existence of a family of operators with these properties implies that $\mathcal{F}(X)$ satifies the
$K$-BAP.
\end{proof}

\section{Conclusion}

We close with a few remarks and unsolved problems. We believe that $K$-random projections can be a useful tool to understand better 
the geometry of the space $\hbox{Lip}(X)$. On the basis of Theorem~\ref{thm:main} it is natural to ask:

\begin{problem}
Give a direct proof that any doubling metric space admits a $K$-random projection on any closed subspace.
\end{problem} 

Example~\ref{main example} shows that not all $K$-random projections are induced by a $K$-gentle partition. 
This leads to the following question.

\begin{problem}
Does there exist a metric space $(X,d)$ and a closed subspace $M$ admitting $K$-random projections, but with no
$K$-gentle partition of unity?  
\end{problem}

Recall that $Z$ has a {\em Schauder basis} if there exists a sequence $(z_n)_n$ in $Z$ such that any $z\in Z$ can be uniquely written as the sum of a norm convergent series $\sum_n a_nz_n$, for some  sequence $(a_n)_n$ of scalars.
It is well known that a Banach space $Z$ admits a Schauder basis if and only if there exists a sequence 
of uniformly bounded linear projections $(P_n)_{n \in \mathbb{N}}$, $P_n: Z \longrightarrow Z$ such that
\begin{enumerate}
\item[($i$)] $P_m \circ P_n = P_{\min\{m, n\}}$ for all $n, \,m \in \mathbb{N}$;
\item[($ii$)] dim$P_1(Z)<\infty$ and $\hbox{\rm dim}(P_{n+1}-P_n)(X)=1$
for all $n \in \mathbb{N}$;
\item[($iii$)] $\overline{\bigcup_{n \in \mathbb{N}} P_n(Z)} = Z$.
\end{enumerate}
In particular, setting $Z=\mathcal{F}(X)$, if one is able to find a uniformly bounded sequence of finite rank projections 
$$T_n: \hbox{Lip}_0(X) \longrightarrow \hbox{Lip}_0(X) $$
which are weak$^*$-weak$^*$ continuous such that $T_m \circ T_n = T_{\min\{m, n\}}$ for all $n,\, m \in \mathbb{N}$, 
which converges to the identity on $\hbox{Lip}_0(X)$, such that  $\hbox{\rm dim}(T_{n+1}-T_n)(\hbox{Lip}_0(X))=1$, then such a sequence gives rise to a sequence $P_n: \mathcal{F}(X) \longrightarrow \mathcal{F}(X)$ with the properties $(i),\, (ii),\,(iii)$ above.

Following Theorem~\ref{BAP}, where we related the asymptotic $K$-random projections to the $K$-BAP property,
one can then ask the following:

\begin{problem}
Let  $(X,d,\bar x)$ be a separable pointed metric space  which admits a $K$-random projection w.r.t. $M_n$ for each $n \in \mathbb{N}$, where $M_n$ are  as in Theorem~\ref{BAP}. 
Does $\mathcal{F}(X)$ admit a Schauder basis? 
\end{problem}

A positive answer to this question would imply that for any doubling metric space $(X,d)$ the space $\mathcal{F}(X)$ admits a Schauder basis.

\end{document}